\documentclass[a4paper]{amsart}
\usepackage{amsmath}
\usepackage{amssymb}
\usepackage[arrow, matrix, curve]{xy}

\usepackage{tikz}

\def\to{\rightarrow}

\def\newDC{DC}

\def\DC{\mathcal B}
\def\Bar{\mathbb B}
\def\coBar{\mathbb B^\vee}
\def\OOmega{\mathbf \Omega}
\def\OOmegao{\mathbf \Omega^\circ_r}
\def\AA{\mathbb A}
\def\AAop{\AA^{op}}
\def\ZZ{\mathbb Z}
\def\DsymC{DC}
\def\Ch{Ch}
\def\sigmamod{\Sigma\textrm{-mod}}

\def\isomap{\xrightarrow{\raisebox{-0.7ex}[0ex][0ex]{$\sim$}}}
\def\isomapgauche{\xleftarrow{\raisebox{-0.7ex}[0ex][0ex]{$\sim$}}}
\def\mono{\rightarrowtail}

\newtheorem{theorem}{Theorem}[section]
\newtheorem{lemma}[theorem]{Lemma}
\newtheorem{prop}[theorem]{Proposition}
\newtheorem{cor}[theorem]{Corollary}

\newtheorem{theorem*}{Theorem}
\newtheorem{prop*}[theorem*]{Proposition}

\theoremstyle{definition}

\newtheorem{example}[theorem]{Example}
\newtheorem{notation}[theorem]{Notation}

\theoremstyle{remark}
\newtheorem{remark}[theorem]{Remark}

\numberwithin{equation}{section}

\begin{document}

  \title{Homology of infinity-operads}

\author[E. Hoffbeck]{Eric Hoffbeck}
\address{Université Sorbonne Paris Nord, LAGA, CNRS, UMR 7539, F-93430, Villetaneuse, France}
\email{hoffbeck@math.univ-paris13.fr}

\author[I. Moerdijk]{Ieke Moerdijk}
\address{Department of Mathematics, Utrecht University, PO BOX 80.010, 3508 TA Utrecht, The Netherlands}
\email{i.moerdijk@uu.nl}
 
\subjclass[2010]{18N70, 55N35, 18M70}

\keywords{}

\thanks{}
%
\begin{abstract}
In a first part of this paper, we introduce a homology theory for infinity-operads and for dendroidal spaces which extends the usual homology of differential graded operads defined in terms of the bar construction, and we prove some of its basic properties. In a second part, we define general bar and cobar constructions. These constructions send infinity-operads to infinity-cooperads and vice versa, and define an adjoint bar-cobar (or "Koszul") duality. Somewhat surprisingly, this duality is shown to hold much more generally between arbitrary presheaves and copresheaves on the category of trees defining infinity-operads. We emphasize that our methods are completely elementary and explicit.
\end{abstract}

\maketitle
 

%

\section*{Introduction}

The goal of this paper is to introduce a homology theory for infinity-operads and dendroidal spaces, and prove some of its fundamental properties. In particular, we will prove a bar-cobar (sometimes called “Koszul”) duality in a general context. Our homology extends the classical homology of differential graded operads, first introduced by Ginzburg-Kapranov \cite{GK} in terms of the bar construction on differential graded operads.

More specifically, we will consider a category $\AA$ of trees, closely related to the category $\OOmega$  \cite{MW1} used to model infinity-operads, but modified to deal with  operads which are trivial in arities zero and one.  We introduce a homology theory for presheaves of chain complexes on $\AA$, which takes values in symmetric sequences of graded abelian groups. This theory enjoys the standard properties of a homology theory such as the invariance under quasi-isomorphism and the existence of long exact sequences and spectral sequences. Furthermore, in the case where the presheaf arises as the nerve of an operad, our theory agrees with the classical bar homology of the operad. For example, in the specific case of the associative operad, our complex calculating the homology can directly be related to the dual of the complex given by the Stasheff polytopes, while for the commutative operad, it is the homology of the partition complex, as in the paper of Fresse \cite{F}. In fact, if the presheaf comes from a general operad in $Sets$ (as in the case of the nerves of the associative and commutative operads), our homology is closely related to the generalised partition complexes of Vallette \cite{V}. If the presheaf arises from the restriction of a dendroidal space to the smaller category $\AA$ of trees, our homology theory is an invariant of the Quillen model structure defining infinity-operads, in the sense that it can be viewed as a left Quillen functor on this model category.

We next prove a duality theorem for arbitrary presheaves on the category $\AA$. In order to do so, we define for a presheaf $M$ on the category $\AA$ of trees a copresheaf $\Bar (M)$, which extends the complex calculating the homology of $M$ mentioned above. We also define a dual cobar construction $\coBar$ which assigns a presheaf to any copresheaf. The duality result then states for any presheaf $M$  that $\coBar (\Bar (M))$ is quasi-isomorphic to $M$ (as presheaves, i.e. naturally in $\AA$). We emphasize that this result holds for arbitrary presheaves $M$, not just for nerves of operads. In this sense, it reflects a property of the category $\AA$. Our proof of duality is completely explicit, elementary and  short. Nonetheless, in the strict case where $M$ is the nerve of an actual operad, we recover the duality result first proved by Getzler-Jones \cite{GJ} and by Ginzburg-Kapranov \cite {GK}, see also  \cite{KS}. Extending what happens in these references, our construction yields an infinity-cooperad $\Bar (M)$ in case the presheaf $M$ has the structure of an infinity-operad, (and dually, the construction $\coBar$ maps infinity cooperads back to infinity-operads.) The relevant notion of cooperad arising naturally here is closely related to that occurring in the papers of Ching \cite{C, C2} and the preprint of Fresse-Guerra \cite{FG}, for example. 

A duality theorem for operads in spectra has been proved earlier by Ching, using the same category $\AA$ of trees. Although Ching works with strict operads and his proof in terms of the Boardman-Vogt resolution is different, it is likely that his methods extend to infinity operads in spectra, and we expect there to be a common generalisation of his results and ours.

We mentioned that our homology can be seen as defining a homology theory for dendroidal sets or spaces, behaving well with respect to the model structure for simplicial or topological infinity operads (see \cite{CM}). In this context, our homology theory should be distinguished from the one defined by Basic and Nikolaus \cite{BN}. Their homology does not extend the homology of operads but extends the usual one of simplicial sets, and in fact is an invariant of their stable model structure modeling connective spectra (see Nikolaus \cite{N}). A concrete difference is that for representable dendroidal sets (or free operads generated by trees) our homology counts the vertices of a tree (see Corollary~\ref{corRepre} below), while theirs counts the leaves.

\textbf{Acknowledgements:} We began the work on this paper when the second author was visiting the Université Paris 13 as "professeur invité" in 2017, and he is grateful to this institution. He would also like to thank Gijs Heuts for helpful discussion. The first author acknowledges support by ANR ChroK (ANR-16-CE40-0003).

%

\section{Definitions}\label{S:Def}

Recall the category $\OOmega$ of rooted trees from \cite{MW1}, parametrising the category $dSets$ of dendroidal sets. This category $\OOmega$ contains many subcategories which are relevant in different contexts. Here we consider two such, viz. the full subcategory $\OOmegao$ of open and reduced trees  and a non-full subcategory $\AA \subset \OOmegao$.
This category $\OOmegao$ has as its objects all trees $T$ in $\OOmega$ all of whose vertices have valence (i.e. number of inputs) at least $2$. The category $\AA$ has the same objects. The morphisms $S \to T$ in $\AA$ are those morphisms in $\OOmega$ which preserve the root and induce a bijection between the sets of leaves. Any such morphism factors as an isomorphism $S \isomap S'$ followed by an inner face map $S' \mono T$, i.e. an ``inclusion'' of a tree $S'$ obtained from $T$ by contracting some inner edges. Note that the category $\AA$  falls apart into disjoint pieces $\AA^{(\ell)}$ of trees with $\ell$ leaves, where the corolla $C_\ell$ is an initial object. For a tree $T$, we denote by $E(T)$ the set of its \textit{inner} edges and by $Vert(T)$ the set of its vertices.
Note that for $T$ with at least one vertex, 
$$|E(T)| \leq \ell(T) -2,$$
where $\ell(T)$ is the number of leaves of $T$.

We will be interested in various categories of presheaves on $\AA$ and $\OOmegao$.

\begin{example}
 (1) Let $P$ be an operad with values in a cocomplete symmetric monoidal category $\mathcal M$. We will mostly be interested in the case where $\mathcal M$ is the category of modules over a commutative ring $R$, or the category $\Ch$ of chain complexes of such (such operads will be called \emph{linear operads}), or the category of simplicial sets, and the reader can keep one of these cases in mind for definiteness.
  The usual dendroidal nerve $N(P)$ defines a presheaf on $\AA$ with values in $\mathcal M$.
  Its value on a tree $T$ can loosely be described as the tensor product 
  $$N(P)(T) = \bigotimes_v P(|v|),$$
  where this tensor product ranges over the vertices $v$ of $T$ and $|v|$ is the number of input edges of $v$ in $T$.
  A more precise definition requires either a definition of operads as taking values $P(U)$ on finite sets $U$ or a quotient of the form 
  $$N(P)(T) = \bigg( \coprod_\pi \bigotimes_v P(|v|)\bigg)  / \sim,$$
  where the sum is over planar structures $\pi$ on $T$, and $|v|$ is now the ordered set of inputs of $v$ linearly ordered by $\pi$, while the quotient is by the group of (non-planar) automorphisms of $T$. We refer to \cite{HM} for details.
  
  Our restriction to the subcategory $\AA \subset \OOmega$ implies that $N(P)$ ignores the objects $P(0)$ and $P(1)$, so this notion is only appropriate for operads which are reduced, that is where $P(0)$ is empty or zero and $P(1)$ consists of the identity operation only.
  In the linear case, $P(1)=R$. This restriction is not essential, as one can always replace the ground ring $R$ by $P(1)$, as in \cite{GK}. 
  
  (2) Recall that a \emph{dendroidal set} (or \emph{dendroidal space}) is a presheaf on the category $\OOmega$ with values in $Sets$ (or in $sSets$, respectively).
Its restriction to a presheaf on $\AA$ or on $\OOmegao$ will still be referred to as a dendroidal set (or space).
 For an open and reduced dendroidal space $X: (\OOmegao)^{op} \to sSets$, we obtain a functor $\ZZ[X]: (\OOmegao)^{op} \to \Ch$ by taking the usual chain complex $\ZZ[X(T)]$ of each simplicial set $X[T]$. This functor can be restricted to the smaller category $\AA \subset \OOmegao$.

 \end{example}

For two trees $S$ and $R$, and a leaf $a$ of $S$, we will write  $S \circ_a R$ for the tree obtained by grafting $R$ on top of $S$ at the leaf $a$.
Thus, the inner edges of $S \circ _a R$ are those of $S$, those of $R$ and $a$.
A functor $M : \AA \to \Ch$ is called a \textit{linear $\infty$-preoperad} if $M$ comes equipped with structure maps
$$\theta = \theta_{S,R,a} : M( S \circ _a R) \to M(S) \otimes M(R)$$
which are natural in $S$ and $R$ (in the sense that for $\beta : S \to S'$ and $\gamma : R \to R'$ in $\AA$, the diagram
\begin{equation*}
\xymatrix@M=8pt{
M(S \circ_a R)\ar[r]^\theta & M(S)\otimes M(R) \\ 
M(S' \circ_{a'} R')\ar[u]\ar[r]^\theta & M(S')\otimes M(R')\ar[u]}
\end{equation*}  
commutes, where $a'=\beta a$), and satisfy the natural associativity axioms (see \cite{LV}).
The functor $M$ is called a \textit{linear $\infty$-operad} if these maps $\theta$ are moreover quasi-isomorphisms.

\begin{example}
(1) If $P$ is a linear operad, then for the nerve of $P$, the maps $\theta$  are isomorphisms. \\
(2) If $X : \AAop \to sSets$ is (the restriction to $\AA$ of) a dendroidal Segal space (see Section~\ref{S:DHom}), then $\ZZ[X]$ is a linear $\infty$-operad.
\end{example}

\bigskip

Similarly, a functor $Y : \AA \to \Ch$ is called a \textit{linear $\infty$-precooperad} if it comes equipped with maps
$$\Delta = \Delta_{S,R,a} : Y(S \circ_a R) \to Y(S) \otimes Y(R)$$
which are natural in $S$ and $R$ and satisfy coassociativity axioms.
The functor $Y$ is then called a \textit{linear $\infty$-cooperad}  if these maps $\Delta$ are moreover quasi-isomorphisms.

%

%

\section{Dendroidal homology}\label{S:DC}

We want to define and study a notion of dendroidal homology, which takes values in symmetric sequences of abelian groups.
Let us first recall the category 
$$\sigmamod = \prod_{ \ell \geq 2} \Ch_{\ZZ[\Sigma_\ell]}$$
of symmetric sequences of chain complexes. 
An object $A$ of $\sigmamod$ is a family $\{A^{(\ell)}\}_{\ell \geq 2}$ where each $A^{(\ell)}$ is a nonnegative chain complex of $\Sigma_\ell$-modules. 
We define
$$\DsymC : \Ch^{\AA^{op}} \to \sigmamod$$
to be the functor assigning to a presheaf $M$ on $\AA$ the symmetric sequence of complexes whose $\ell$-th component in degree $p\geq-1$ is
$$\DsymC_p(M)^{(\ell)} = \Bigg( \bigoplus_{T, \alpha, e} M(T) \Bigg)_{coinv}$$
where
\begin{itemize}
\item $T$ ranges over trees with exactly $p+1$ inner edges and at least one vertex
\item $\alpha : C_\ell \mono T$ is a morphism in $\AA$
\item $e=(e_0, \ldots, e_p)$ is an enumeration of the inner edges of $T$
\end{itemize}
and
the coinvariants are taken for the action of the groupoid $H$ given by the semi-direct product of
the groupoid of extensions $C_\ell \mono T$ and isomorphisms between them 
\begin{equation*}
\xymatrix@M=10pt{
C_\ell \ar@{>->}[dr]_{\alpha'} \ar@{>->}[r]^{\alpha} & T \ar[d]_\theta^{\wr} \\
 & T' }
\end{equation*}  
and the symmetric group $\Sigma_{[p]}=\Sigma_{\{0, \ldots, p\}}$ (permuting the $e_i$'s and acting by the sign representation).
Notice that the morphism $\alpha$ is the same as an enumeration of the leaves of $T$ and commutativity of the diagram above corresponds to $\theta$ respecting this enumeration.
We write elements as (represented by) quadruples
$$(T,\alpha,e,x)$$
where $e$ is as above and $x \in M(T)$.
As we take the coinvariants with respect to the $H$-action, we make the following two identifications: First,
$$(T,\alpha, e, \theta^* x')=(T',\theta \alpha,\theta e, x')$$
for $x' \in M(T')$, $\theta : T \isomap T'$ and $\theta e$ the induced enumeration of the inner edges of $T'$. And secondly
$$(T,\alpha,e,x)=(-1)^\tau (T,\alpha,e\tau,x)$$
for any permutation $\tau \in \Sigma_{[p]}$ and the composition $[p]\stackrel{\tau}{\to} [p] \stackrel{e}{\to} E(T)$.
The $\Sigma_\ell$-action is given by precomposition, identifying $\Sigma_\ell$ with the automorphisms of $C_\ell$.

The simplicial face map $\partial_i : \newDC_p(M) \to \newDC_{p-1}(M)$ maps the summand $M(T)$ for $(T,\alpha,e)$ to $M(\partial_{e_i}T)$ for $(\partial_{e_i}T, \partial_{e_i}^* (\alpha), (e_0, \ldots, \hat{e_i}, \ldots, e_p))$, where the map $\partial_{e_i}: \partial_{e_i}(T) \mono T$ is the face contracting $e_i$, the map $\partial_{e_i}^* (\alpha): C_\ell \mono \partial_{e_i}(T)$ is the unique morphism in $\AA$ for which $\alpha= \partial_{e_i} \circ \partial_{e_i}^* (\alpha)$,  and the map $M(T) \to M(\partial_{e_i}T)$ is the restriction $\partial_{e_i}^*$ given by the presheaf structure on $M$. 

Let us check that these face maps are well-defined on the quotient.
First, we have the following commutative square which shows that the action by the groupoid of tree isomorphisms is compatible with the differential:
\begin{equation*}
\xymatrix@M=10pt{
T'  \ar@{->}_\theta^\sim[r] & T \\
\partial_{e_i}T' \ar@{>->}[u]^{\partial(e_i)} \ar@{->}^\sim[r] & \partial_{\theta(e_i)}T. \ar@{>->}[u]_{\partial_{\theta(e_i)}} }
\end{equation*}  
Next, let us show the compatibility of the $\Sigma_{[p]}$-action. 
For a given tree $T$  and a map $\alpha: C_\ell \to T$, let us write $M(T)_e$ for the copy of $M(T)$ given by $T$ and the enumeration $e=(e_0, \ldots, e_p)$. 
Then for a given $e$ and $i=0, \ldots, p$, and any $\tau \in \Sigma_{[p]}$, there is a commutative square
\begin{equation*}
\xymatrix@M=10pt{
M(T)_e   \ar@{->}^{(-1)^\tau}[r] \ar@{->}[d]_{(-1)^j \partial_{e_j}^*} 
& M(T)_{e \tau} \ar@{->}[d]^{(-1)^i \partial_{(e\tau)_i}^*} \\
 M(\partial_{e_j}T) \ar@{->}_{(-1)^{\tau'}}[r] 
& M( \partial_{(e\tau )_i}T)  }
\end{equation*}  
where $j =\tau i$, and $\tau'$ is the restriction of  $\tau$ 
(so that  $(-1)^{\tau'}=(-1)^{i+j}(-1)^\tau$ for $i$ and $j$ as above).

The simplicial identities are clearly satisfied.
Moreover, 
the complex is naturally augmented by 
 $$\newDC_{-1}(M) = \bigoplus_{\ell \geq 2} M(C_\ell)$$
 where $C_\ell$ is the $\ell$-corolla. 
Since the simplicial operators clearly commute with the differential of $M$, this proves that $\newDC(M)$ is an augmented semi-simplicial object in $\sigmamod$.
One obtains a double complex 
$$\bigoplus_{p,q}\newDC_p(M_q)$$
where $p \geq -1$ and $q \geq 0$,
by taking as horizontal differential the alternating sum of the face maps
$$\partial=(-1)^i \partial_{e_i}:\newDC_p(M) \to \newDC_{p-1}(M)$$
and vertical differential $\partial_{int}$ induced by the one of $M$,
$$\partial_{int}(T,\alpha,e,x)=(-1)^{p+1}(T,\alpha,e,\partial_M(x)).$$
The homology of the total complex will be denoted 
$$DH_*(M)$$
(with $* \geq -1$), and called the \emph{dendroidal homology} of $M$.

\section{Examples and basic properties}

In this section we relate our homology theory to the usual homology of operads and to the homology of categories.

\begin{example}
 Let $P$ be a linear (reduced) operad. We have seen that its nerve $N(P)$ is a presheaf on $\AA$ with values in $\Ch$.
 The homology of $P$ (as defined for instance by Ginzburg-Kapranov \cite{GK}) is
 the homology of its bar construction $B(P)=(T^c(s\bar P), \partial)$, the quasi-cofree cooperad on the suspension of the augmentation ideal of $P$, with a differential which contracts one edge at a time. Up to a degree shift, the underlying symmetric sequence (in arities larger than $1$) is nothing but our complex defined above.
Thus the homology of $P$ is (up to a shift) the dendroidal homology of $N(P)$.
\end{example}

\begin{example}
As a particular case of the above remark, for $P=Ass$ the associative operad, 
our complex allows us to directly compute the dendroidal homology of $N(Ass)$ and thus recover the homology of $Ass$.
The complex $DC_*(\ZZ[NAss])$ decomposes as a sum over $\ell \geq 2$, the number of leaves. For a fixed $\ell$ and in degree $p$, it is the free abelian group on planar trees with $\ell$ leaves and $p+1$ inner edges ($p \geq -1$) equipped with an enumeration of the leaves. For a fixed enumeration, this complex is dual to that of the cellular homology of the Stasheff polytope, which is a contractible CW-complex of dimension $\ell-2$. So normally the homology would be $\ZZ[\Sigma_\ell]$ in degree $\ell-2$ and $0$ elsewhere. But our grading is shifted by $-1$, so the homology of the summand $DC_*^{(\ell)}(\ZZ[NAss])$ is $\ZZ[\Sigma_\ell]$ in degree $p = \ell-3$ and zero elsewhere. 

Similarly, for $P=Com$ the commutative operad, our complex allows us a direct computation, see also Remark~\ref{RemOpPartPoset} below.
\end{example}

\begin{prop}\label{propbasicprop}
Dendroidal homology has the following basic properties.\\
 (i) If $M \to M'$ is a quasi-isomorphism, then the induced map $DH_*(M) \to DH_*(M')$ is an isomorphism.\\
  (ii) For a filtered colimit $M=\displaystyle{\lim_\alpha M_\alpha}$, the canonical map 
  $\displaystyle{\lim_\alpha DH_*(M_\alpha) \to DH_*(M)}$ is an isomorphism.\\
 (iii) For any short exact sequence $0 \to M' \to M \to M'' \to 0$, 
 there is a long exact sequence
 $$ \ldots \to DH_n(M') \to DH_n(M) \to DH_n(M'') \to DH_{n-1}(M') \to \ldots$$
 (iv) For any presheaf $M$ on $\AA$ with values in $Ch$, its homology $DH_*(M)$ decomposes according to the number $\ell$ of leaves in a tree, as
  $$DH_*(M)=\bigoplus_{\ell\geq2} DH_*^{(\ell)}(M).$$
  This decomposition is functorial in $M$.\\
 \end{prop}

 \begin{proof}
(i) First, for each $p \geq -1$, the map $\newDC_p(M_\bullet) \to \newDC_p(M'_\bullet)$ is a quasi-isomorphism, and this gives a quasi-isomorphism of double complexes and hence a quasi-isomorphism between the total complexes. \\
(ii)  Obvious. \\
  (iii) This follows from the fact that such a short exact sequence induces a short exact sequence of double complexes
  $$0 \to \newDC_p(M'_q) \to \newDC_p(M_q) \to \newDC_p(M''_q) \to 0.$$
  (iv) The double complex $\newDC_*(M)$ already decomposes according to the number of leaves, and the differentials respect this decomposition.
\end{proof}

\begin{remark}
Let us observe that the definition of $DH_*(M)$ as the homology of the total complex of a double complex (bigraded by $p \geq -1$ and $q \geq 0$) gives two spectral sequences which we record for completeness:
 For $M$ a presheaf on $\AA$ with values in $\Ch$, there are natural spectral sequences 
 $$E^2_{p,q} = DH_p(H_q(M)) \Rightarrow DH_{p+q}(M)$$ 
 and
 $$E^2_{p,q}=H_q(DH_p(M)) \Rightarrow DH_{p+q}(M).$$
 For the convergence of these, even if $M$ is unbounded, note that for a fixed number of leaves $\ell$, the spectral sequences live in a strip $ -1 \leq p \leq \ell-3$.
\end{remark}

\begin{prop}\label{PropRepre}
For a representable functor $\AA(-,R)$ given by a tree $R$,
 $$DH_p^{(\ell)}(\ZZ \AA(-,R))= 
\left\{\begin{array}{ll}
\ZZ[\Sigma_\ell] & \text{if } p=-1 \text{ and } R \simeq C_\ell \\ 
0 & \text{otherwise.}\\
\end{array}\right. $$
 \end{prop}

\begin{proof}
 The complex in degree $p$ is the free abelian group on equivalence classes of
 $$ C_\ell  \stackrel{\alpha}\mono  T  \stackrel{\beta}\mono  R, \ (e_0, \ldots, e_p)$$
 where $e_0, \ldots, e_p$ enumerate inner edges of $T$
 and the equivalence relation is for permutations of the $e_i$'s (with a sign) as well as isomorphisms given by commutative diagrams
\begin{equation*}
\xymatrix@M=8pt
{
C_\ell \ar@{>->}[r]^\alpha \ar@{>->}[dr]_{\alpha'} & 
T \ar@{>->}[r]^\beta \ar[d]^{\wr}_\theta&
R \\
 & T'\ar@{>->}[ur]_{\beta'} & 
}
\end{equation*}  
all as above. The differential is $\displaystyle{\sum_{i=0}^p(-1)^i\partial_i}$
where
$$\partial_i(C_\ell  \stackrel{\alpha}\mono  T  \stackrel{\beta}\mono  R, \ (e_0, \ldots, e_p)) = (C_\ell \mono  \partial_{e_i}T  \mono  R, \ (e_0, \ldots, \hat{e_i}, \ldots, e_p)
)$$
with the induced maps.
The composition $\beta \alpha$ remains the same over equivalence classes and under $\partial_i$, thus the complex falls apart into a direct sum over morphisms
$\gamma: C_\ell \mono R$. 
For a fixed $\gamma$, choose a linear ordering of the inner edges of $R$. Then the summand for $\gamma$ can be identified with the complex which in degree $p$ is the free abelian group on the set of sequences
$$e_0 < \ldots < e_p \ \ (p \geq -1) $$
of such edges in the chosen order, with usual differential
$$\partial(e_0, \ldots, e_p)=\sum(-1)^i (e_0, \ldots, \hat{e_i}, \ldots, e_p).$$
If $R$ has at least one inner edge, the complex is acyclic: we can choose a maximal element in the order, call it $m$.
Then the map $h$ defined by
$$h(e_0, \ldots, e_p) =
\left\{\begin{array}{ll}
0 & \text{if } e_p=m \\ 
(-1)^{p+1}(e_0, \ldots e_p, m) & \text{if } e_p <m.\\
\end{array}\right.
$$
is a contracting homotopy, so the homology vanishes.
If $R$ is a corolla, on the other hand, the complex is concentrated in degree $-1$, and is just $\ZZ$. So
$$DH_{-1}(\ZZ \AA(-,R))=\bigoplus_{C_\ell \isomap R} \ZZ,$$
which is $\ZZ[\Sigma_\ell]$ when $R$ is a corolla with $\ell$ leaves and $0$ otherwise.
\end{proof}

\begin{cor}\label{corRepre}
For a representable functor $\OOmegao(-,T)$ given by a tree $T$, 
 $$DH_p(\ZZ \OOmegao(-,T))= 
\left\{\begin{array}{ll}
\displaystyle{ \bigoplus_{v \in Vert(T)} \ZZ[\Sigma_{inputs(v)}] }& \text{if } p=-1 \\ 
0 & \text{otherwise.}\\
\end{array}\right.
 $$
  \end{cor}
 
 \begin{proof}
 It suffices to notice that the restriction of $\OOmegao(-,T)$ to $\AA \subset \OOmegao$ decomposes as a sum of representables. Writing $j:\AA \to \OOmegao$ for the inclusion,
  $$j^* \OOmegao(-,T) = \coprod_{R \subset T} \AA(-,R)$$
  where the coproduct ranges over external faces $R \mono T$ of $T$.
   Indeed, any map $S \to T$ in $\OOmegao$ decomposes uniquely as a map $S\mono R$ in $\AA$ followed by an external face $R \mono T$.
   As external faces $R \to T$ where $R$ is a corolla are in bijective correspondence with vertices of $T$, the result now follows from the
   previous proposition.
\end{proof}

It is possible to express dendroidal homology in terms of the homology of a pair of categories.
Recall that for a category $\mathbb C$ and a \emph{covariant} functor $A : \mathbb C \to Ab$, the homology $H_*(\mathbb C, A)$ is defined as the homology of the chain complex
$$C_n(\mathbb C, A)= \bigoplus_{c_0 \to \ldots \to c_n} A(c_0)$$
where the sum is taken over strings of morphisms $c_0 \to \ldots \to c_n$ in $\mathbb C$.
The differential is $\partial = \sum (-1)^i \partial_i$ where $\partial_i$  maps the summand for  $c_0 \to \ldots \to c_n$ to the one for 
 $c_0 \to \ldots \to \hat{c_i} \to \ldots \to c_n$ by the identity of $A(c_0)$  if $i \neq 0$ and by $A(c_0)\to A(c_1)$ if $i=0$.
 If $\mathbb D \subset \mathbb C$ is a subcategory, the homology of the pair
 $H_*(\mathbb C, \mathbb D ; A)$ is defined to be that of the quotient complex
 $$0 \to C_*(\mathbb D, A) \to C_*(\mathbb C, A) \to C_*(\mathbb C, \mathbb D ; A) \to 0,$$
as usual.

If $A$ is a \emph{contravariant} functor on $\mathbb C$, we write $C_*(\mathbb C, A)$ for
$C_*(\mathbb C^{op}, A)$, and similarly for $H_*(\mathbb C, A)$ and for pairs. 

\begin{prop}\label{PropHomCat}
For a given $\ell \geq 2$, consider the category $C_\ell / \AA$ and its full subcategory $C_\ell /\!/ \AA$ on the non-isomorphisms $C_\ell \mono T$.
Then for $M:\AAop \to Ab$ and its restriction to $C_\ell/\AA$,
$$DH^{(\ell)}_*(M)=H_{*+1}(C_\ell/\AA, C_\ell/\!/\AA ; M).$$
\end{prop}

\begin{proof}
 Consider the double complex (with values in chain complexes)
 $$C_{p,q} = \bigg( \bigoplus_{ C_\ell \mono R \mono T_0 \mono \ldots \mono T_q} M(T_q) \bigg) / \sim$$
 where $q \geq 0$, $p\geq-1$, $C_\ell \stackrel{e}\mono R \mono T_0 \mono \ldots \mono T_q$ is in $\AA$, 
 $e=(e_0, \ldots, e_p)$ enumerates inner edges of $R$,
 and we make identifications for permutations of the $e_i$'s (with signs) and isomorphisms of the kind
 \begin{equation*}
\xymatrix@M=8pt@R=6pt
{ & R \ar@{>->}[dr] \ar[dd]^{\wr} &&& \\
C_\ell \ar@{>->}[ur] \ar@{>->}[dr] & 
. &
T_0 \ar@{>->}[r] & \ldots  \ar@{>->}[r]& T_q. \\
 & R'\ar@{>->}[ur] & 
}
\end{equation*}
The differential in the $p$-direction is the alterning sum of $\partial_i$'s, as defined in the previous proof. The vertical differential in the $q$-direction is the simplicial one on the part $T_0 \mono \ldots \mono T_q$. 

For $p$ and $C_\ell \stackrel{e}\mono R$ fixed, the complex computes the homology of $R/\AA$ with coefficients in $M$, so this contracts to $M(R)$. Since the quotient is by a simple groupoid (see Appendix), we obtain that $H_q(C_{p,-})=0$ for $q>0$ and $H_q (C_{p,-})=\newDC_p(M)$ if $q=0$.
We conclude that the total complex computes $DH_p(M)$.

On the other hand, for $q$
fixed, we get a sum over $C_\ell \stackrel{\gamma}\mono T_0 \mono \ldots \mono T_q$ of complexes computing the dendroidal homology of the representables 
$\AA(-, T_0)$. This vanishes when $\gamma$ is not an isomorphism (ie. when $\gamma$ is a morphism in $C_\ell /\!/ \AA$), as we just saw,
so for fixed $q$
$$H_p(C_{-,q}) = 
\left\{\begin{array}{ll}
\displaystyle{ \bigoplus_{C_\ell \isomap T_0 \mono \ldots \mono T_q} M(T_q) } & \text{if } p=-1 \\ 
0 & \text{if } p \neq -1.\\
\end{array}\right.
$$
So the spectral sequence of the double complex with 
$$E^1_{p,q} = H_p(C_{-,q})$$
collapses, and the differential $d^1:E^1_{p,q} \to E^1_{p,q-1}$ 
is induced by the vertical differential. 
Thus 
$(E^1_{-1,q},d^1)$ is the complex computing the homology 
$H_{\ast}(C_\ell/\AA, C_\ell/\!/\AA ; M).$
\end{proof}

\begin{remark}
A similar statement of course follows for a presheaf $M$ with values in chain complexes and the corresponding hyperhomology of a pair of categories.
\end{remark}

\begin{remark}
There is a long exact sequence induced by the short exact sequence of complexes
 $0 \to C_*(C_\ell /\!/ \AA ; M) \to C_*(C_\ell / \AA ; M) \to C_*(C_\ell / \AA, C_\ell /\!/ \AA ; M) \to 0.$
Moreover $C_\ell / \AA$ has an initial object. Thus
$DH_*^{(\ell)}(M)$ is $H_*(C_\ell /\!/ \AA ; M)$ for $* \geq 1$.
\end{remark}

\begin{remark}\label{RemOpPartPoset}
If $M$ is a constant presheaf, then in positive degrees $DH_*^{(\ell)}(M)$ is the homology of the classifying space of $C_\ell /\!/\AA$, which is related to the partition poset of $\{1, \ldots, \ell \}$. This applies in particular to the nerve of the commutative operad. This partition poset is homotopy equivalent to a wedge of spheres, see Robinson-Whitehouse~\cite{RW}. 
Operadic partition posets (see \cite {V}) can be recovered as $C_\ell/\!/\AA / N(P)$, that is the category of objects strictly under the corolla and over $N(P)$.

\end{remark}

\section{Homology of dendroidal spaces}\label{S:DHom}

For $X$ an open reduced dendroidal space, we denote by $DH_*(X)$ the homology of $\newDC_*(\ZZ[X])$.  
The results of the previous section imply the following corollaries.

\begin{cor}
For an open and reduced dendroidal space $X$ and for a pushout along a monomorphism
\begin{equation*}
\xymatrix@M=10pt{
U  \ar@{->}[r] \ar@{>->}[d] & X \ar@{>->}[d]\\
V \ar@{->}[r] & Y }
\end{equation*}  
 there is a natural long exact Mayer-Vietoris sequence
 $$ \ldots \to DH_n(U) \to DH_n(V) \oplus DH_n(X) \to DH_n(Y) \to DH_{n-1}(U) \to \ldots.$$\\
 \end{cor}
 
 \begin{proof}
 This follows from Prop~\ref{propbasicprop} (iii) because the pushout yields a short exact sequence
 $$0 \to \ZZ[U] \to \ZZ[V] \oplus \ZZ[X] \to \ZZ[Y] \to 0.$$
 \end{proof}
 
 \begin{cor}
 For an open and reduced dendroidal set $X$, there is a natural isomorphism
$$DH_*(X)=H_{*+1}(\AA/X, \AA_{>0}/X; \ZZ).$$
where $\AA/X$ is the category of elements of $X$ and $\AA_{>0}$ is the full subcategory of $\AA$ on trees with at least one inner edge.
 \end{cor}

The categories $dSets$ and $dSpaces$ of dendroidal sets and spaces carry various Quillen model structures.
The main ones are the operadic model structure on $dSets$ and the model structure for complete Segal dendroidal spaces on $dSpaces$. These two model structures are Quillen equivalent to each other, as well as to the natural model structure on simplicial operads.
We refer to \cite{CM} or \cite{HM} for precise statements and further details.
Similar statements hold for presheaves on the smaller category $\OOmegao$ of open and reduced trees. These categories of presheaves are denoted 
$ordSets$ and $ordSpaces$ -- open and reduced dendroidal sets and spaces.
The Quillen model structures mentioned above restrict to these smaller categories, and make them Quillen equivalent to each other and to the category of simplicial operads $P$ which are open (no nullary operation) and reduced (the spaces of unary operations are weakly contractible). 

The remarks above motivate the study of the homology of objects in $ordSpaces$. The relevant model structure is the localisation of the generalised Reedy model structure, localised by the so-called Segal maps
$$S \cup_a T \mono S \circ_a T$$
for two trees $S$ and $T$ and the grafting $S \circ_a T$ of $T$ onto a leaf $a$ in $S$.
We identify trees $T$ with the corresponding representable functors $\OOmegao(-,T)$, viewed as discrete dendroidal spaces, open and reduced if $T$ is.
The cofibrations in this category are the normal monomorphisms: those  monomorphisms $X \mono Y$ for which each automorphism group $Aut(T)$ of a tree $T$ acts freely on the complement of the image of $X(T)$ in $Y(T)$. The fibrant objects in this model category are the $\infty$-operads:
Reedy fibrant objects $X$ of $ordSpaces$ for which the Segal maps $X(S \circ_e T) \to X(S) \times_ {X(\eta)} X(T)$ are all weak equivalences, see \cite {CM}.

The theory of Quillen model categories enables us to express the fundamental properties of dendroidal homology in a possibly more informative way.
The category $\sigmamod$ carries a projective model structure, in which the weak equivalences and fibrations are induced from those of chain complexes, and in which the cofibrations are the maps $A \to B$ which are monomorphisms with projective cokernel (i.e., each $A^{(\ell)} \mono B^{(\ell)}$ has a projective $\ZZ[\Sigma_\ell]$-module as a cokernel).

\begin{theorem}
 The functor $\DsymC : ordSpaces \to \sigmamod$ is a left Quillen functor with respect to the model structures mentioned above.
\end{theorem}

\begin{proof}
 We need to check that the functor $\DsymC$ has a right adjoint, that it maps Reedy cofibrations to cofibrations and Reedy weak equivalences to weak equivalences (i.e., quasi-isomorphisms),
 and that is sends the localising trivial cofibrations of the form
 \begin{equation} \tag{$\ast$}
 (S\cup_a T) \boxtimes V \cup (S \circ_a T) \boxtimes U \to (S\circ_a T) \boxtimes V \label{Eq1}
 \end{equation}
 to weak equivalences.
 Here $S \circ_a T$ is a grafting as above, and $U \mono V$ is a monomorphism of simplicial sets.
 Moreover, for a dendroidal set $D$ and a simplicial set $U$, we write $D \boxtimes U$ for the dendroidal space $D \boxtimes U(T)_n=D(T) \times U_n = \coprod_{d \in D(T)} U_n$.
 
 Now, first of all, $\DsymC$ preserves all small colimits, hence has a right adjoint for general categorical reasons.
 This right adjoint $\mathcal R : \sigmamod \to ordSpaces$ is defined by
 $$\mathcal R(A)(T)_n = Hom(\DsymC(T \boxtimes \Delta[n]),A)$$
 where the $Hom$ is that of $\sigmamod$. See the remark below for an elaboration of this description of $\mathcal R$.
 
 Next, for a Reedy cofibration $X \mono Y$,  i.e. a normal monomorphism, the cokernel of 
 $\DsymC(X)^{(\ell)} \to \DsymC(Y)^{(\ell)}$ in bidegree  $(p,q)$ is
 $$\Bigg( \bigoplus_{\alpha:C^\ell \mono T,e} \ZZ [Y(T)_q \setminus X(T)_q]\Bigg)_{coinv}$$
 where
 $\alpha$ ranges over morphisms in $\AA$ 
 and $e=(e_0, \ldots, e_p)$ are enumerations of the inner edges of $T$,
 while we have identified $X(T)$ with its image in $Y(T)$.
 The coinvariants are by isomorphisms under $C_\ell$ (as in the first diagram of Section~\ref{S:DC}) 
 and permutations of the $e_i$'s.
 The action by $\theta \in \Sigma_\ell$ sends a generator $(T,\alpha,e,y)$ to $(T,\alpha \theta,e,y)$. 
 
 We claim that this action is free if $X \mono Y$ is normal.
 Indeed, if $(T,\alpha,e,y) = (T,\alpha \theta,e,y)$, then there exists a (unique) factorisation
\begin{equation*}
\xymatrix@M=10pt{
C_\ell  \ar@{>->}[r] & T \\
C_\ell \ar@{->}[u]_{\theta} \ar@{>->}[r]& T \ar@{-->}[u]_{\tau} }
\end{equation*}  
 for which $y=\tau^* y$. 
 If $y$ does not belong to the image of $X$ then $\tau$ is the identity by normality, and hence so is $\theta$. This shows that $\DsymC$ preserves cofibrations.
 
 Next, $\DsymC$ preserves Reedy weak equivalences. Indeed, if $X \mono Y$ has the property that each $X(T) \to Y(T)$ is a weak equivalence of simplicial sets, then the first item of Prop~\ref{propbasicprop} shows that $\DsymC(X) \to \DsymC(Y)$ is a homology isomorphism.
 
 Finally, we need to check that $\DsymC$ maps each trivial cofibration of the form (\ref{Eq1})
above to a quasi-isomorphism.
But in a fixed simplicial degree $q$, this map is a coproduct of copies of the identity on
$S \circ_a T$ and copies of the inclusion ${S\cup_a T} \to {S \circ_a T}$.
Since such an inclusion induces a bijection on vertices, the previous proposition on the Mayer-Vietoris sequence together with Corollary~\ref{corRepre} shows that for a fixed $q$, this map induces an isomorphism in homology.
Thus, by the spectral sequence, $\DsymC$ sends (\ref{Eq1}) to a quasi-isomorphism.
\end{proof}

\begin{remark}
Recall for a non-negative chain complex $B$ the Dold-Kan simplicial set $R(B)$ defined by
$$R(B)_n=Hom(N_*(\Delta[n]),B),$$
where $N_*(\Delta[n])$ is the normalised chain complex of $\Delta[n]$.
If $\Sigma_\ell$ acts on $B$, then it also acts on $R(B)$. In particular, $R$ can be viewed as a functor from $\sigmamod$ to symmetric sequences of simplicial sets. Now
$$\newDC(T \boxtimes \Delta[n])_m= \bigoplus_{p+q=m-1} \newDC_p(T) \otimes \ZZ[\Delta[n]]_q$$
By Corollary~\ref{corRepre}, the right hand side is quasi-isomorphic as a symmetric collection to
$$\bigoplus_{|v|=\ell} \ZZ[C_\ell] \otimes N_*(\Delta[n])$$
where $v$ ranges over vertices of $T$ and $|v|$ is the number of inputs of $v$.
So for $A$ is in the proof above
$$\mathcal R(A)(T)\simeq \prod_{|v|=\ell} R(A^{(\ell)}).$$

\end{remark}

\begin{remark}
 This theorem is false for the covariant and Picard model structures, see \cite{HM} or \cite{BN} where the Picard model structure is called the basic model structure.
 Indeed, one readily checks that for the tree $T=C_v \circ_e C_w$ with a single inner edge connecting a bottom vertex $v$ to a top vertex $w$, the ``top horn'' $\Lambda^v(T)=  \partial_e(T) \cup \partial_w(T) \mono T$ does not induce an isomorphism in homology (it is the map $\ZZ \oplus \ZZ \to \ZZ \oplus \ZZ$, which is the identity on the first component and zero on the second component).
\end{remark}

%

%

\section{Bar  construction}\label{S:DCold}
The goal of this section is to extend the bar  construction of operads to the present context, and give a reformulation in terms of the homology of categories. 

For a presheaf $M : \AAop \to \Ch$, we define a \textit{covariant} functor
$\DC_*(M)(-)$ from $\AA$ to double chain complexes .
For $S \in \AA$, define
$$\DC_p(M)(S) = \Bigg( \bigoplus_{\alpha : S \mono T, d,e} M(T) \Bigg)/H_S,$$
where
\begin{itemize}
\item $\alpha : S \mono T$ is a morphism in $\AA$
\item $d=(d_0, \ldots, d_{q-1})$ is an enumeration of the inner edges of $S$
\item $e=(e_0, \ldots, e_{p-1})$ is an enumeration of the inner edges of $T$ that do not belong to the image of $\alpha$,
\end{itemize}
and $H_S$ is the groupoid naturally acting on these data: it is the semi-direct product of the groupoid of extensions $S \mono T$ of $S$
and isomorphisms between them
\begin{equation*}
\xymatrix@M=10pt{
S \ar@{>->}[dr]_{\alpha'} \ar@{>->}[r]^{\alpha} & T \ar[d]_\theta^{\wr} \\
 & T' }
\end{equation*}  
and the symmetric groups $\Sigma_{q}=\Sigma_{\{0, \ldots, q-1\}}$ and $\Sigma_{p}$.

We will write elements of $\DC_p(M)(S)$ as (represented by) triples $$(\alpha, d|e, x)$$ where $\alpha,d,e$ are as above and $x \in M(T)$. We omit the tree $T$ from the notation to keep it simpler.
When necessary, we will write these elements as 
$$(\alpha, (d_0, \ldots, d_{q-1})|(e_0, \ldots, e_{p-1}), x).$$
The fact that $\DC_p(M)(S)$ is defined as coinvariants for $H_S$ comes down to the following identifications:
$$(\alpha, d|e, \theta^* x')=(\theta \alpha, d| \theta e, x')$$
for $x' \in M(T')$ and $\alpha, d,e, \theta$ as above, and
$$(\alpha, d|e, x)=(-1)^{\sigma}(-1)^{\tau}(\alpha, d\sigma|e\tau, x)$$
for any permutations
$\sigma \in \Sigma_{q}$ and $\tau \in \Sigma_{p}$, and where
$d \sigma$ is
the composition $\{0, \ldots, q-1\} \xrightarrow{\sigma} \{0, \ldots, q-1\} \xrightarrow{d} E(T)$ and $e\tau$ is defined in a similar way.

\begin{notation} \label{Not:Signs}
In the formula for $\partial_{ext}$ below, we write $(-1)^d$ for $(-1)^{q}$ with $d$ representing here the length of the sequence rather than the sequence itself. In the rest of the paper, we will often use this notation when writing signs.
\end{notation}

The first differential
$$ \partial_{ext}: \DC_p(M)(S) \to \DC_{p-1}(M)(S) $$
is defined  by
$$\partial_{ext} (\alpha, d|e, x) = \sum_{i=0}^{p-1} (-1)^{d+i} (\partial_{e_i}^* (\alpha), d|(e_0 \ldots \hat{e_i} \ldots e_{p-1}), \partial_{e_i}^* (x)),$$
where $\partial_{e_i} :\partial_{e_i}(T) \to T$ is the morphism in $\AA$ given by contracting the edge $e_i$, 
the map $\partial_{e_i}^* (\alpha): S \mono \partial_{e_i}(T)$ is the unique morphism in $\AA$ for which $\alpha= \partial_{e_i} \circ \partial_{e_i}^* (\alpha)$, 
and $\partial_{e_i}^*: M(T) \to M (\partial_{e_i} (T))$ is given by the presheaf structure on $M$.
This map  $\partial$ is a well-defined map and squares to zero.
The second differential $\partial_{int}$ is induced by the one of $M$ and the Koszul sign rule. 
More precisely, for $(\alpha, d|e, x)$ in $\DC_p(M)(S)$ and $x \in M_n(S)$,
$$\partial_{int}(\alpha, d|e, x)=(-1)^{d+e} (\alpha, d|e, \partial_M x.)$$
These two differentials anti-commute.

\begin{notation} 
We write $\DC_{p,n}(M)(S)$ for the double complex, with differentials $\partial_{ext}$ and $\partial_{int}$. 
Later in Notation~\ref{Not:Bar} we will consider a shifted total complex.
\end{notation}

\begin{remark}
As the category $\AA$  falls apart into pieces $\AA^{(\ell)}$ of trees with $\ell$ leaves, so does the functor $\DC_*(M)$.
Moreover, we can recover $\DsymC$ from $\DC$, in the following way:
$$\DsymC_*(M)= \prod_{\ell \geq 2} \DC_{*+1}(M)(C_\ell).$$
(Notice the shift, which as we will see later, is motivated by an induced structure of cooperad, see~\ref{BarIsCoop}.)
\end{remark}

\begin{remark}
Using the result of the appendix, the same complex can be given in terms of invariants:
$$\bar\DC_p(M)(S)= \Bigg( \prod_{S \mono T,d|e} M(T) \Bigg)^{inv}$$
Notice that this expression is obviously covariantly functorial in $\AA$:
for $\omega \in \bar\DC_p(M)(S)$, a morphism $\gamma: S \to S'$ in $\AA$, $\alpha:S' \mono T$, $d$ an enumeration of the inner edges of $S'$ and $e$ an enumeration of the inner edges of $T$ not in the image of $\alpha$,
the induced map $\bar \gamma_* : \bar\DC_p(M)(S) \to \bar\DC_{p-c}(M)(S')$ is given by
$$\bar \gamma_*(\omega)_{\alpha,d|e}=
(-1)^\tau \omega_{\alpha \gamma, d'|d''e}$$
where 
$\tau$ is a permutation acting on the enumeration $d$ such that $\tau d=(d',d'')$
with $d'$ an enumeration of the edges of $S$ (which can be seen as edges of $S'$ via $\gamma$) 
and $d''$ the edges in $S'$ ``contracted by $\gamma$''. Notice that $\bar \gamma^*$ preserves the degree if we take the degree of $\bar \DC_p(M)(S)$ to be $p+q$ where $q$ is the number of inner edges of $S$ (i.e. ``$d+e$'' in Notation~\ref{Not:Signs}).
\end{remark}

In terms of the definition by coinvariants, the map $\gamma_*:\DC_p(M)(S) \to \DC_{p-c}(M)(S')$ obtained by covariant functoriality from a morphism $\gamma: S \to S'$ in $\AA$ can be described as follows.
For an element of $\DC_p(M)(S)$ represented by $(\alpha, d|e, x)$ as above,
$\gamma_*(\alpha, d|e, x)=0$ if $\alpha: S \to T$ does not factor through $\gamma : S \to S'$. 
If, on the other hand, it does factor, we can write $\alpha = \beta \gamma$
for a unique $\beta$.
Then, up to isomorphism, 
$S$ is obtained from $S'$ by contracting a number $c$ of edges in $S'$. 
We refer to this number $c$ as the codimension of $\gamma$, $c=codim(\gamma)$.
The tree $S'$ is obtained by contracting $p-c$ edges of $T$.
Up to a permutation $\tau$, we can suppose that, among the edges $(e_0, \ldots, e_{p-1})$, the edges $e_0, \ldots, e_{c-1}$ are the edges in the image of $\beta$, while
$e_c, \ldots, e_{p-1}$ are the edges ``contracted by $\beta$''.
Then
$$\gamma_*(\alpha, d|e, x)=(\beta, (\gamma d_0, \ldots, \gamma d_{q-1}, \beta^{-1}e_0, \ldots, \beta^{-1}e_{c-1}) | (e_c, \ldots, e_{p-1}) , x).$$
This renders the following diagram commutative, where $\rho$ is the isomorphism of the appendix.
\begin{equation*}
\xymatrix@M=10pt{
S \ar@{->}[d]_\gamma 
& \DC_p(M)(S) \ar@{-->}[d]_{\gamma_*} \ar@{->}[r]^{\rho}_\sim 
& \bar\DC_p(M)(S) \ar@{->}[d]_{\bar\gamma_*}   \\
S'&  \DC_{p-c}(M)(S')
\ar[r]^\rho_\sim & \bar\DC_{p-c}(M)(S').}
\end{equation*} 
The description by invariants gives as a consequence: 
\begin{itemize}
\item For $\gamma : S \mono S'$, the map $\gamma_*$ as described above is well-defined on equivalence classes.
\item For $\gamma : S \mono S'$ and $\gamma' : S' \mono S''$, we have $(\gamma' \circ \gamma)_* = \gamma'_* \circ \gamma_*.$ 
\end{itemize}
Moreover, the map $\gamma_*$ commutes with $\partial_{ext}$.
This shows $\DC_*(M)$ is a well-defined functor from $\AA$  to double chain complexes.
Since the differential $\partial_{ext}$ commutes with that of $M$ up to sign, taking the homology of $\DC(M)$ with respect to $\partial_{ext}$ gives a new functor 
$H_*\DC(M) : \AA \to \Ch$.

\begin{remark}
Using the notation of the appendix, both differentials can also be described from the point of view of the invariants.
The map $\bar \partial_{ext}$ from $\bar \DC_p(M)(S)$ to $\bar \DC_{p-1}(M)(S)$ is given by
$$(\bar\partial_{ext})_{S \mono T, d|e} = (-1)^{d} \int_{T\mono T'} f^* (\omega_{S \mono T',d|fe})$$
where the integral is over the groupoid of codimension 1 extensions $f:T \mono T'$ of $T$, where
$f$ is also the name of the unique edge in $T'\setminus T$, and where the map $S \mono T'$ is given by the composition of $f$ with the map $\alpha:S\mono T$.
The map $\bar \partial_{int}$ from $\bar \DC_{p,n}(M)(S)$ to $\bar \DC_{p,n-1}(M)(S)$ is given by
$$(\bar\partial_{int})_{S \mono T, d|e} = (-1)^{d+e} \partial_M (\omega_{S \mono T,d|e}).$$
Both signs here can be interpreted as Koszul signs.

\end{remark}

%

The following two propositions can be proven in the same way as the analogous statements for $\newDC$ (see Proposition~\ref{PropRepre} and Proposition~\ref{PropHomCat}).

\begin{prop}
 For $R$ a tree and $M = \ZZ[\AA(-,R)]$, we have
 $$H_p\DC(M)(S)= 
\left\{\begin{array}{ll}
\ZZ[Iso(S,R)] & \text{if } p=0 \text{ and } S \simeq R \\ 
0 & \text{otherwise.}\\
\end{array}\right.
 $$
\end{prop}

\begin{prop}
For $S \in \AA$ fixed, consider the category $S / \AA$ and its full subcategory $S /\!/ \AA$ on the non-isomorphisms $S \mono T$, i.e. maps of positive codimension.
Then for $M:\AAop \to Ab$ and its restriction to $S/\AA$, there is an isomorphism
$$H_*\DC(M)(S)\simeq H_{*}(S/\AA, S/\!/\AA ; M).$$
\end{prop}

%

We next observe that our construction sends $\infty$-operads to $\infty$-cooperads (up to a suspension), as for to the usual operadic bar construction.

\begin{notation}\label{Not:Bar}
For $M : \AAop \to \Ch$ and $S$ a tree, we denote by $\Bar(M)(S)$ the total complex defined by 
$$\Bar_m(M)(S)= \bigoplus_{m=p+q+n} \DC_{p,n}(sM)(S)$$
where $q$ is the number of inner edges of $S$ and where $sM$ denotes the suspension of $M$, with the convention $sx \in (sM)_n$ for $x \in M_{n-1}$.
Notice that for a morphism $\gamma:S \to S'$ the induced map $\gamma_* :\Bar(M)(S) \to \Bar(M)(S')$ preserves the total degree.
\end{notation}

\begin{prop}\label{BarIsCoop}
 Let $M : \AAop \to \Ch$ be an $\infty$-operad. Then the functor $\Bar(M) : \AA \to \Ch$ is naturally an $\infty$-cooperad.
\end{prop}

The proof is based on the following lemmas.

\begin{lemma}
 Given a grafting $S \circ_a R$, there is a bijective correspondence between isomorphism classes of extensions
 $S \circ_a R \stackrel{\alpha}\mono T$ in $\AA$, and pairs of such $(S \stackrel{\alpha_S}\mono T_S, R \stackrel{\alpha_R}\mono T_R)$. 
 Under this correspondence, $T=T_S \circ_{a'} T_R$ where $a'=\alpha(a)$, and $codim(\alpha)= codim(\alpha_S)+codim(\alpha_R)$.
\end{lemma}

\begin{proof}
This is an immediate consequence of the fact that any injective morphism in $\OOmega$ factors uniquely (up to isomorphism) as a composition of inner face maps
followed by a composition of outer face maps (see~\cite{HM}). These factorisations are the dotted arrows in the following diagram:
\begin{equation*}
\xymatrix@M=6pt@R=14pt@C=14pt{
S \ar@{>->}[r] \ar@{>.>}[d] & S\circ_a R \ar@{>->}[d] & R \ar@{>->}[l] \ar@{>.>}[d]\\
T_S\ar@{>.>}[r] & T & T_R \ar@{>.>}[l]}
\end{equation*}  
Note that the horizontal maps are in $\OOmegao$ but not in $\AA$.
\end{proof}

A corollary of this lemma is the following.
\begin{lemma}\label{Lemm59}
Given a grafting $S \circ_a R$ as above, there is a canonical quasi-isomorphism:
 $$\Bigg( \prod_{S \mono T_S,d|e} M(T_S) \Bigg)^{inv} \otimes \Bigg( \prod_{R \mono T_R,d|e} M(T_R) \Bigg)^{inv}
 \isomap \Bigg( \prod_{S \circ_a R \mono T,d|e} M(T) \Bigg)^{inv}.$$ 
 Similarly, there is a quasi-isomorphism for the descriptions by coinvariants:
 $$\Bigg( \bigoplus_{S \mono T_S,d|e} M(T_S) \Bigg)_{coinv} \otimes \Bigg( \bigoplus_{R \mono T_R,d|e} M(T_R) \Bigg)_{coinv}
 \isomapgauche \Bigg( \bigoplus_{S \circ_a R \mono T,d|e} M(T) \Bigg)_{coinv}.$$ 
\end{lemma}

\bigskip

\begin{proof}[Proof of the proposition]
 Suppose $M$ is a linear $\infty$-operad. We define a map 
 $$\Delta: \Bar(M)(S\circ_a R) \to \Bar(M) (S) \otimes \Bar(M) (R),$$ coming from maps
 $$\DC_{p,n}(sM)(S\circ_a R) \to \bigoplus_{p_1+p_2=p, n_1+n_2=n+1} \DC_{p_1,n_1}(sM)(S) \otimes  \DC_{p_2,n_2} (sM)(R)$$
 as follows.
 An element of $\DC_{p,n}(sM)(S\circ_a R)$ can be represented as 
 $$(S \circ_a R \stackrel{\alpha}{\mono} T, d|e, sx)$$
 where $\alpha$ is a morphism of codimension ${p}$, while $d$ enumerates the inner edges of $S \circ_a R$
 and $e=(e_0, \ldots, e_{p-1})$ those of $T$ which do not lie in the image of $\alpha$, and finally $x \in M_{n-1}(T)$. 
 By permuting the sequence  $d$ (and adapting the sign if necessary), we may assume that $d$ enumerates $a$ first, then the edges of $S$,
 and those of $R$ last, as 
 $$d=(a,d_S,d_R).$$
 For factorisations $S \stackrel{\alpha_S}\mono T_S$ and $R \stackrel{\alpha_R}\mono T_R$ 
 as in the lemma, we can take $T_S$ and $T_R$ to be actual subtrees of $T$,
 and the $p_1$ inner edges of $T_S$ not in the image of $\alpha_S$ together with the $p_2$ inner edges of $T_R$ not in the image of $\alpha_R$
 are exactly the inner edges of $T$ not in the image of $\alpha$. 
 So, by permuting the $e_i$'s if necessary, we can assume that $e=(e_S,e_R)$ enumerates the inner edges of $T_S$ before those of $T_R$.
 
 The operad structure map $\Delta$ on $\Bar(M)$ is induced by the operad structure of $M$ (shifted by $1$) and the degrafting of $S$ and $R$: for an element $(\alpha, d|e, sx)$ in $\DC_{p,n}(sM)(S \circ_a R)$, 
 $$\Delta(\alpha, d|e, sx)=\sum \pm (\alpha_S, d_S|e_S, s\theta_{(1)}(x)) \otimes (\alpha_R, d_R|e_R, s\theta_{(2)}(x)) $$
 where $\pm$ is a Koszul sign explained below, and where (in Sweedler notation)
 $$\theta(x)=\sum  \theta_{(1)}(x) \otimes  \theta_{(2)}(x) \in M_{n_1-1}(S) \otimes M_{n_2-1}(R)$$
 is given by the operad structure of the presheaf $M$, with the degrees $n_1$ and $n_2$ satisfying $n-1=(n_1-1)+(n_2-1)$.
 The sign in the formula is defined as a Koszul sign. Let us first define the map $\tilde\theta: sM \to sM \otimes sM$ as $(s \otimes s) \theta s^{-1}$. Thus $\tilde\theta$ is a map of degree $+1$ and  $\tilde\theta(sx)= \sum (-1)^{n_1-1}  s\theta_{(1)}(x) \otimes  s\theta_{(2)}(x)$ with the same notations as above. The map $\Delta$ is the composition of applying $\tilde\theta$ on $sx$ followed by degrafting the trees $S$ and $R$ (which involves permuting $d_R$ with $e_S$) followed by permuting $s\theta_{(1)}(x)$ with $(\alpha_R, d_R|e_R)$ of degree $d_R+e_R$. The obtained Koszul sign is then 
 $$1+d+e+e_Sd_R+(e_R+d_R)n_1+n_1-1.$$
 Another way to understand this sign is to write
 $$\Delta(\alpha,(a,d_S,d_R)|(e_s,e_R),sx)= \pm tw((\alpha_S,d_S|e_S),(\alpha_R,d_R|e_R), \tilde\theta(sx))$$
 where $\pm$ is the Koszul sign $(-1)^{1+d_S+e_S+d_R+e_R+e_Sd_R}$ and $tw$ moves the first component of $\tilde\theta(sx)$ in the right place (with the appropriate sign) in order for the result to land in the appropriate tensor product.

 We leave it to the reader to check that $\Delta$ is well-defined and satisfies the required naturality and coassociativity conditions, keeping in mind that $\tilde\theta$ is anti-associative.
 Moreover, $\Delta$ is compatible with the differentials; i.e., using the Koszul sign convention, the following diagram commutes,
\begin{equation*}
\xymatrix@C=40pt{
\DC_{p}(sM)(S \circ_a R)\ar[rr]^{\Delta \quad \quad \ \quad \quad } \ar[d]^{\partial_{ext}} 
&& \displaystyle{\bigoplus_{p_1+p_2=p}} \DC_{p_1}(sM)(S) \otimes \DC_{p_2}(M)(R) \ar[d]^{\partial}  \\ 
\DC_{p-1}(sM)(S \circ_{a} R)\ar[rr]^{\Delta \quad \quad \ \quad \quad } 
&&\displaystyle{\bigoplus_{r_1+r_2=p-1}} \DC_{r_1}(sM)(S) \otimes \DC_{r_2}(M)(R)
}
\end{equation*}  
  where the vertical map on the right hand side is
 $\partial_{ext} \otimes id +(-1)^{d_S+e_S +n_1} id \otimes \partial_{ext}$. The same diagram also commutes for the differential $\partial_{int}$.
 Moreover, using the lemma, one can easily check that $\Delta$ is a quasi-isomorphism if $\theta$ is.
\end{proof}

%

\section{Cobar construction}

Let $Y : \AA \to \Ch$ be a covariant functor. We will define a contravariant
functor $\DC^\vee : \AAop \to \Ch$, dual in a sense to the construction of $\DC(M)$ above.
For an object $R$ of $\AA$, we define
$$\DC^\vee_q(Y)(R) = \Bigg( \prod_{\alpha : R \mono S, b,d} Y(S) \Bigg)^{inv},$$
where  the product ranges over morphisms $R \stackrel{\alpha}\mono S$ of codimension $q$, enumerations $b=(b_0, \ldots, b_{r-1})$ of inner edges of $R$, and enumerations $d=(d_0, \ldots, d_{q-1})$ of the inner edges of $S$ that do not belong to the image of $\alpha$.
The invariance is for isomorphisms under $R$,
\begin{equation*}
\xymatrix@M=10pt{
R \ar@{>->}[dr]_{\alpha'} \ar@{>->}[r]^{\alpha} & S \ar[d]^\theta_\wr \\
& S' }
\end{equation*}  
and permutations of the $b_i$'s and $d_j$'s. 
Thus an element $\omega \in \DC^\vee_q(Y)$ assigns to every $\alpha, b,d$ as above
an element 
$$\omega_{\alpha,b|d} \in Y(S),$$
subject to the  following invariance conditions:
\begin{itemize}
 \item for an isomorphism $\theta$ as above,
 $$\theta_*(\omega_{\alpha,b|d})=\omega_{\theta\alpha,b|\theta d}$$
 \item for $\sigma \in \Sigma_{r}$ and the composition $b\sigma:\{0, \ldots, r-1\} \to E(R)$,
 $$\omega_{\alpha,b|d}=(-1)^{|\sigma|} \omega_{\alpha,b\sigma|d}$$
 \item for $\tau \in \Sigma_{q}$ and $d\tau:\{0, \ldots, q-1\} \to E(S)$,
 $$\omega_{\alpha,b|d}=(-1)^{|\tau|} \omega_{\alpha,b|d\tau}$$
\end{itemize}
The first differential 
$\partial_{ext} : \DC^\vee_q(Y)(R) \to \DC^\vee_{q+1}(Y)(R)$
is defined by 
$$(\partial_{ext}\omega)_{\alpha, b|d} = \sum_{j=0}^{q-1} (-1)^{q-1-j} (\partial_{d_j})_*
(\omega_{\partial_{d_j}^*(\alpha),b|(d_0 \ldots \hat{d_j} \ldots d_{q-1})})$$
where $\partial_{d_j}^*(\alpha)$ is defined by 
\begin{equation*}
\xymatrix@M=10pt{
R \ar@{>->}[dr]_{\partial_{d_j}^*(\alpha)} \ar@{>->}[r]^{\alpha} & S \\
 & \ar[u]^{\partial_{d_j}} \partial_{d_j} S }
\end{equation*}  
as before.
One readily checks that $\partial_{ext}\omega$ is invariant whenever $\omega$ is,
and that $\partial_{ext}\partial_{ext}\omega=0$.

Just like for the bar construction, $\DC^\vee_q(Y)(R)$ is actually a double complex, with an internal differential $\partial_{int}$ defined by
$$(\partial_{int}\omega)_{\alpha, b|d} = (-1)^{b+d}\partial_Y(\omega_{\alpha, b|d}).$$

\begin{remark}
 
For $\DC^\vee$ it is easy to define the differential, but more involved to define functoriality.
Thus as in the case of the $\DC$ construction, we first give another description by coinvariants,
$$\bar\DC^\vee_p(Y)(R)= \Bigg( \bigoplus_{\alpha : R \mono S, b,d} Y(S) \Bigg)_{coinv},$$
isomorphic to $\DC^\vee_p(Y)(R)$.
From this point of view, the functoriality is easier to define.
Let $\beta: R' \mono R$ be a morphism of codimension $c$.
For a map $\alpha:R \to S$, and $b,d,y$ as above, let us define
$$\bar\beta^*(\alpha:R\mono S, b|d,y)=(-1)^\tau (\alpha\beta, b'|\alpha(b'')d,y)$$
where $b'$ are the edges in the image of $\beta$, which we view as edges in $R'$, and $b''$  the edges of $R$ contracted by $\beta$, while $\tau$ the permutation sending the enumeration $b$ to the concatenation of $b'$ and $b''$.
\end{remark}

We now describe the functoriality in terms of invariants, that is the map $\beta^*$ which fits into the following commutative square, where $\rho$ is the isomorphism of the appendix:
\begin{equation*}
\xymatrix@M=10pt{
R '\ar@{->}[d]_\beta
& \bar\DC^\vee_{q+c}(Y)(R')  \ar@{->}[r]^{\rho}_\sim 
& \DC^\vee_{q+c}(Y)(R')    \\
R&  \bar\DC_{q}^\vee(Y)(R) \ar@{->}[u]_{\bar\beta^*}
\ar[r]^\rho_\sim & \DC^\vee_{q}(Y)(R)\ar@{-->}[u]_{\beta^*}.}
\end{equation*} 
For $\omega \in \DC^\vee_q(Y)(R)$ and $\alpha' : R' \mono S$, $b'$ denoting the edges in $R'$ and $d$ those in $S$ not in the image of $\alpha'$, 
$$\beta^*(\omega)_{\alpha', b'|d}=0$$
if $\alpha': R' \to S$ does not factor through $\beta: R' \to R$.
If it does, we can write $\alpha'=\alpha \beta$ for a unique $\alpha$,
where (up to isomorphism) $\alpha$ contracts $q$ edges among the $d_j$'s, and we can split the sequence $d$ into two sequences, first the edges of $S$ in the image of $R$, followed by the remaining edges of $S$ (in order of occurrence in $d$). In other words, 
$$d\tau=\alpha(b)d'$$
for a permutation $\tau \in \Sigma_{q+c}$. Then
$$\beta^*(\omega)_{\alpha', b'|d}= (-1)^ {|\tau|} \omega_{\alpha, \beta(b) \alpha(b)|d'}.$$

A direct verification shows:
\begin{itemize}
  \item The map $\beta^*$ commutes with $\partial_{ext}$.
  \item For a composition $R'' \stackrel{\beta'}\mono R' \stackrel{\beta}\mono R$, we have $\beta'^* (\beta)^*= (\beta \beta' )^*$.
 \end{itemize}

\begin{remark} 
The differentials of $\bar\DC^\vee(Y)$ can be described in terms of coinvariants as follows:
For $(\alpha:R\mono S, b|d,y)$ where $y \in Y(S)$,
$$\bar \partial_{ext}(\alpha:R\mono S, b|d,y)= \int_{f:S\mono S'} (f\alpha:R \mono S', b|df, f_*y)$$
where the integral is over codimension $1$ extensions $f:S\mono S'$, where $f$ also denotes the only new edge, and $f_* y$ is the effect of the covariant functoriality $Y(S) \to Y(S')$ on $y \in Y(S)$.
This renders the following diagram commutative:
\begin{equation*}
\xymatrix@M=10pt{
\bar\DC^\vee(Y)(R)  \ar@{->}[r]^{\rho}_\sim \ar@{->}[d]_{\bar \partial_{ext}}
& \DC^\vee(Y)(R) \ar@{->}[d]^{ \partial_{ext}}   \\
  \bar\DC^\vee(Y)(R) 
\ar[r]^\rho_\sim & \DC^\vee(Y)(R).}
\end{equation*} 
The second differential is given by 
$$\bar \partial_{int}(\alpha:R\mono S, b|d,y)=(-1)^{b+d}(\alpha:R\mono S, b|d,\partial_Y y).$$

\end{remark}

\begin{notation}\label{Not:CoBar}
 For $Y : \AA \to \Ch$ and $R$ a tree, we denote by $\coBar(Y)(R)$ the total complex defined by
 $$\coBar_m(Y)(R) = \bigoplus_{m=n-(r+q)} \DC^\vee_{q,n}(s^{-1}Y)(R)$$
 where $r$ denotes the number of inner edges of $R$.
\end{notation}

The following proposition is the dual of Proposition~\ref{BarIsCoop}.

\begin{prop}
If $Y : \AA \to \Ch$ is a linear $\infty$-cooperad, then $\coBar(Y): \AAop \to \Ch$ has the structure of a linear $\infty$-operad. 
\end{prop}

\begin{proof}
We define the structure with the point of view of coinvariants, to make it closer to the case of the bar construction.
We define a map 
 $$\Delta:\coBar(Y)(S\circ_a R) \to \coBar(Y)(S) \otimes \coBar(Y)(R),$$ coming from maps
 $$\bar\DC^\vee_{q,n}(s^{-1}Y)(S\circ_a R) \to \bigoplus_{q_1+q_2=q, n_1+n_2=n-1} \bar\DC^\vee_{q_1,n_1}(s^{-1}Y)(S) \otimes  \bar\DC^\vee_{q_2,n_2} (s^{-1}Y)(R)$$
 as follows.
 An element of $\bar\DC^\vee_{q,n}(s^{-1}Y)(S\circ_a R)$ can be represented as 
 $$(S \circ_a R \stackrel{\alpha}{\mono} T, ad_Sd_R|e_Se_R, s^{-1}y).$$

 The operad structure map $\Delta$ on $\coBar(Y)$ is induced by the cooperad structure of $Y$ (shifted by $-1$) and the degrafting of $S$ and $R$: for an element $(\alpha, ad_Sd_R|e_Se_R, s^{-1}y)$ in $\bar\DC^\vee(s^{-1}Y)(S \circ_a R)$, 
 $$\Delta(\alpha,  ad_Sd_R|e_Se_R, s^{-1}y)=\pm tw(\alpha_S, d_S|e_S,\alpha_R, d_R|e_R, \hat\theta(s^{-1}y))$$
 where $\pm$ is a sign explained below, $\hat\theta$ is the map of degree $-1$ defined by the composition $(s^{-1} \otimes s^{-1})\theta s$, and where $tw$ is the map moving the first component of $\hat\theta(s^{-1}y)$ in the right place and then cutting the tensor product.
 The sign $\pm$ is defined by $$(-1)^{(d_S+e_S)(d_R+e_R)+d_Re_S}.$$ 
One can easily check that this map $\Delta$ is natural in $R$ and $S$, and is coassociative, and that the differentials $\bar\partial_{ext}$ and $\bar\partial_{int}$ are coderivations with respect to $\Delta$. Moreover, $\Delta$ is a quasi-isomorphism if $\theta$ is.
\end{proof}

\begin{remark}
Note that the sign for $\Delta$ does not seem to be a Koszul sign. However it fits the properties stated at the end of the proof and the adjunction in the next section.
\end{remark}

%

%

\section{Twisting morphisms}

We develop in our context the theory of twisting morphisms, to obtain an adjunction between bar and cobar.
See \cite{LV} for an account in the usual cases of algebras and linear operads.

The first goal is to show there is a bijective correspondance between maps
$$\varphi : \bar \coBar(Y) \to M$$
in the category of functors from $\AAop$ to $\Ch$
and 
$$\psi : Y \to \bar \Bar(M) $$
in the category of functors from $\AA$ to $\Ch$
and ``twisting cocycles'', which can be either written as
$$\hat\tau: Y(S) \to \bigg(\prod_b M(S)\bigg)^{inv}$$
or as 
$$ \tau : \bigg(\bigoplus_b Y(S)\bigg)_{coinv} \to M(S).$$
Here $b$ enumerates edges of $S$, and changes in $b$ act by sign representation as before.
In a first step, we ignore differentials (so take everything with values in graded abelian groups, for example).

Let us look at $\varphi$, given by
$$\varphi_R: 
\Bigg( \bigoplus_{\beta : R \mono S, b|d} s^{-1}Y(S) \Bigg)_{coinv}
\to M(R)$$
natural in $R$.
An element $(\beta, b |d,s^{-1}y)$ on the left is the restriction along $\beta$ of an element where $\beta$ is the identity:
$$(\beta, b |d,s^{-1}y)=\bar \beta^* (id : S\to S, bd|-,s^{-1}y)$$
so
$$\varphi_R(\beta, b |d,s^{-1}y)= \bar\beta^* \varphi_S (id_S, bd|-,s^{-1}y).$$
Hence $\varphi$ is completely determined by what it does on elements of the form $(id_S, b|-,s^{-1}y)$,
which gives the map 
$$ \tau_S : \bigg(\bigoplus_b Y(S) \bigg)_{coinv}
\to M(S)$$
defined by $\tau_S(b,y)=\varphi_S(id_S,b|-, s^{-1}y)$.
Conversely, the map $\varphi$ can be recovered as  $\varphi_R(\beta,b|d, s^{-1}y)=\beta^*\tau_S(bd,y)$.

The same applies to $\psi : Y \to
 \bar\DC(sM)$, given by
 $$\psi_S : Y(S) \to \Bigg(\prod_{\alpha : S \mono T, d|e} sM(T)\Bigg)^{inv}.$$
For an element $\omega$ on the right,
$$ \omega_{\alpha,d|e}= \bar \alpha_*(\omega_{id_T, de|-})$$
so
$$\psi_S(y)_{\alpha,d|e}=\psi_T(\alpha_*y)_{id_T, de|-}$$
by naturality.
Thus $\psi$ is completely determined by the values $\psi_T(y)_{id_T, de|-}$
which gives the map
$$\hat\tau_S : Y(S) \to \bigg(\prod_e M(S)\bigg)^{inv},$$
defined by $\hat\tau_S(y)(e)=s^{-1}\psi_S(y)_{id_S, e|-}$,
another form of the same twisting cocycle. Conversely, $\psi$ can be recovered as
$\psi_S(y)_{\alpha,d|e}=s\hat\tau_T(\alpha_*y)(de)$.

Now let us involve the differentials and find a condition on $\tau$, respectively $\hat\tau$.
We need to check that the condition on $ \tau$ imposed by $\varphi$ preserving the differential is the same as the one on $\hat\tau$ imposed by $\psi$ preserving the differential.
Let us look at $\varphi$ first.
\begin{equation*}
\xymatrix@M=10pt@C=80pt{
\bigg( \bigoplus\limits_{\beta : R \mono S, b|d} s^{-1}Y(S) \bigg)_{coinv}
  \ar@{->}[r]^{\varphi_R} \ar@{->}[d]_{\bar \partial_{ext} + \bar \partial_{int}}
& M(R) \ar@{->}[d]^{\partial_M}   \\
\bigg( \bigoplus\limits_{\beta : R \mono S, b|d} s^{-1}Y(S) \bigg)_{coinv} 
\ar[r]^{\varphi_R} & M(R).}
\end{equation*} 
By naturality, it is enough to describe the condition for this diagram to commute
for $\beta=id$:
$$\partial_M \varphi_R(id_R,b|-,s^{-1}y) = \varphi_R (\bar \partial_{ext} (id_R, b|-,s^{-1}y)) +  \varphi_R( \partial_{int} (id_R,b|-, s^{-1}y)).$$
In other words,
$$\partial_M  \tau_R (b,y)
= 
\int_{d:R\mono R'} \varphi_R(d:R\mono R', b|d, d_* s^{-1}y) +(-1)^b  \tau_R(b; s\partial_{s^{-1}Y} s^{-1}y),
$$ 
or equivalently,
$$ \int_{d:R\mono R'} d^* \tau_{R'}(bd, d_* y)= \partial_M \tau_R (b,y) + (-1)^{b} \tau_R(b; \partial_Y y),
 $$
 where $d$ denotes at the same time the inclusion morphism $R \mono R'$ and the added edge to obtain $R'$ from $R$. This equation can be seen as a Maurer-Cartan equation.
 
 Now let us look at $\psi$:
\begin{equation*}
\xymatrix@M=10pt@C=80pt{
Y(S)   \ar@{->}[r]^{\psi_S} \ar@{->}[d]^{\partial_Y} & 
\bigg( \prod\limits_{\alpha : S \mono T, d|e} sM(T) \bigg)^{inv}
 \ar@{->}[d]_{\bar \partial_{ext} + \bar \partial_{int}} \\
Y(S)  \ar@{->}[r]^{\psi_S} &
 \bigg( \prod\limits_{\alpha : S \mono T, d|e} sM(T) \bigg)^{inv}.}
\end{equation*} 
Again by naturality, $\psi_S \partial_Y$ and $(\bar \partial_{ext} + \bar \partial_{int}) \psi_S$  are completely determined by their values on elements $(id,d|-)$.
So the commutativity of the diagram is equivalent to the condition that for all $y \in Y(S)$
$$ \psi_S(\partial_Y y)_{(id_S,d|-)}
=(\bar \partial_{ext} \psi_S (y))_{(id_S,d|-)} + \bar \partial_{int} \psi_S(y)_{(id_S,d|-)},
$$ 
$$= (-1)^d \int_{f:S\mono S'} f^* (\psi_S(y)_{(f:S\to S',d|f)}) +  (-1)^d \partial_{sM}(\psi_S(y)_{(id_S,d|-)}) $$ 
$$= (-1)^d \int_{f:S\mono S'} f^* s(\hat\tau_{S'}(f_*y)(df)) +  (-1)^{d+1} s \partial_{M}(\hat\tau(y)(d)) $$ 
or equivalently,
$$ \int_{f:S\mono S'} f^* (\hat\tau_{S'}(f_*y)(df))= \partial_{M}(\hat\tau_S(y)(d))  + (-1)^d \hat\tau_S(\partial_Y y)(d),
 $$
 which is the same Maurer-Cartan equation as above. 
 This proves the following result.
 
\begin{prop}\label{Twisting}
There is a bijective correspondence
 $$\{\varphi :  \coBar(Y)\to M \} 
 \ \ \longleftrightarrow \ \ 
 \{ \psi : Y\to \Bar(M)\}
 \ \ \longleftrightarrow \ \
 \left\{\hat\tau : Y \to \Big(\prod M\Big)^{inv}\right\}$$
 where the $\hat\tau$'s are the twisting cocycles satisfying the above Maurer-Cartan equation.
\end{prop}

 We now consider the unit and the counit of the adjunction: $\eta : Y \to \Bar(\coBar(Y))$
 and $\varepsilon :  \coBar(\Bar(M)) \to M$. 
 The explicit formulas are the following. First, the counit is given by
 $$\varepsilon_S(\alpha,a|e, s^{-1}\omega)=\alpha^*(s^{-1}(\omega_{id_T, ae|-}))$$ 
 where $\alpha: S\mono T$, $a$ is an enumeration of the inner edges of $S$ and $e$ an enumeration of the inner edges of $T$ not in the image of $\alpha$, and $\omega$ is in $\Bar(M)(T)$. Another way to write $\varepsilon$ using only coinvariants is the following
 $$\varepsilon_S(\alpha, a|e, s^{-1}(\beta, ae|f,sx))=
 \left\{
 \begin{array}{ll}
0 & \text{if } codim(\beta) > 0 \\
\alpha^*\beta^*(x) & \text{if } codim(\beta) = 0
 \end{array}
\right.
 $$
 where $\beta: T \mono R$ and $f$ is an enumeration of the inner edges of $R$ not in $T$.
 
 For the unit,
 $$\eta_S (y)_{\alpha,a|e}=s(id_T, ae|-, s^{-1}\alpha_*y)$$
 where $y$ is in $Y(S)$, and $\alpha$, $a$, $e$ are as above.
 An equivalent description of $\eta$, with the point of view of coinvariants for both bar and cobar, is 
 $$\eta_S(y)= \int_a (id_S,a|-, s(id_S,a|-, s^{-1}y)).$$
 For a second proof of Proposition~\ref{Twisting}, the reader might wish to check that these maps $\eta$ and $\varepsilon$ are natural, compatible with differentials and satisfy the triangle identities.

 \begin{prop}
 The functors $\Bar$ and $\coBar$ are part of an adjunction between $\infty$-operads and $\infty$-cooperads. 
 \end{prop}

 \begin{proof}
We need to prove the following four facts. First, the bar construction sends morphisms of $\infty$-operads to morphisms of $\infty$-cooperads. Secondly, the cobar construction sends morphims of $\infty$-cooperads to morphisms of $\infty$-operads. Thirdly, $\varepsilon$  is a map of $\infty$-operads. Lastly, $\eta$ is a map of $\infty$-cooperads.

The first two facts follow directly from the definition of the (co)operad structure in the (co)bar contruction.
Let us now prove that $\varepsilon$ is a map of $\infty$-operads. To make the computation more readable, an element of the form $(\alpha, a|e,x)$ will be denoted by $(a|e,x)$ and we let the reader write the signs using the previously given formulas. In the following computation, we consider a grafting $S \circ_u R$ mapped into a tree $T=T_S \circ T_R$, $a$ and $b$ denote respectively enumerations of the inner edges of $S$ and $R$, and $e$ and $f$ denote respectively enumerations of the inner edges of $T_S \setminus S$ and $T_R \setminus R$, and finally $x$ is in $M(T)$. 
$$\begin{array}{rcl}
\Delta_{\bar\coBar} (uab|ef,s^{-1}(uabef|-,sx))&=&
\pm tw(a|e, b|f, \Delta_{\Bar}(s^{-1}(uabef|-,sx)))\\
&=&\pm tw(a|e, b|f, (s^{-1} \otimes s^{-1}) tw (ae|-, bf|-, \tilde\theta(sx)))\\
&=&\pm tw(a|e, b|f, (s^{-1} \otimes s^{-1}) tw (ae|-, bf|-, sx_1 \otimes sx_2))\\
&=& + (a|e, s^{-1} (ae|-,sx_1))  \otimes ( b|f, s^{-1}( bf|-, sx_2)).
\end{array}$$
This is mapped by $\varepsilon \otimes \varepsilon$ to $e^*(x_1) \otimes f^*(x_2)=(ef)^*\theta(x)$ (where $e$ and $f$ denote here the maps $S \mono T_S$ and $R \mono T_R$, and $ef$ the associated map $S\circ R \to T$)
which is nothing but $\theta\varepsilon(uab|ef,s^{-1}(uabef|-,sx))$ by naturality of $\theta$. This concludes the proof for $\varepsilon$.
The proof that $\eta$ is a map of $\infty$-cooperads can be written in a very similar way.

 \end{proof}

\section{Duality theorem}\label{S:Duality}

In the previous sections, we discussed constructions of 
$\Bar(M) : \AA \to \Ch$ from $M:\AAop \to \Ch$, and of
$\coBar(Y) : \AAop \to \Ch$ from $Y:\AA \to \Ch$.
These constructions map linear $\infty$-operads to linear $\infty$-cooperads, 
and vice versa.
The purpose of this section is to prove the following theorem.

\begin{theorem}
 For any two functors $M:\AAop \to \Ch$ and $Y:\AAop \to \Ch$, the unit and counit $\eta : Y \to \Bar \coBar (Y)$ and $\varepsilon : \coBar \Bar (M) \to M$ are quasi-isomorphisms.
\end{theorem}

Before embarking on the proof of the theorem, let us recall some definitions and simplify some notation.
Let us fix an object $R \in \AA$ and work with the representations by coinvariants for both bar and cobar. Elements in $\Bar \coBar (Y)(R)$ are represented in the form
$$(a|e,s(b|d,s^{-1} y))$$
where $R\stackrel{\alpha}{\to} S\stackrel{\beta}{\to} T$, $a$ enumerates the inner edges of $R$, $e$ those of $S$ not in the image of $\alpha$, $b$ those of $S$, and $d$ those of $T$ not in the image of $\beta$, while finally $y \in Y(T)$.
The total degree of such an element is $|y|-d$.
Note that any equivalence class can be represented in the form where $b=ae$, as we will assume from now on.
Now, for the sake of this proof, let us fix $a$ and simplify notation by deleting $a$ and the (de)suspensions from the notation, so write
$$(e,d,y) \text { for } (a|e,s(ae|d,s^{-1} y)).$$
In this notation, the unit $\eta_R$ can be written as
$$\eta_R(y)=(-,-,y)=(a|-,s(a|-, s^{-1} y)).$$

\begin{proof}
 Up to a degreeshift, $\Bar \coBar (Y)(R)$ is a double complex $C_{p,q}$ with elements
 $$(e,d,y)=((e_0, \ldots, e_{p-1}), (d_0, \ldots, d_{q-1}),y)$$
 and with exterior differential $\partial_{ext}$ of bidegree $(-1,1)$ given by
 $$\partial_{ext}(e,d,y)=\sum_{i=0}^{p-1}(-1)^i ((e_0, \ldots, \hat e_i, \ldots e_{p-1}), (e_i,d_0, \ldots, d_{q-1}),y)$$
 and an internal one $\partial_{int}$ of bidegree $(0,1)$ (where we ignore the degree of $y$). We can picture this complex as
 
\begin{equation*}
\xymatrix{
C_{0,0} \\
C_{0,1}\ar[u]^{\partial_{int}} & C_{1,0}\ar[l]^{\partial_{ext}} \\
C_{0,2}\ar[u]^{\partial_{int}}  & C_{1,1}\ar[u]^{\partial_{int}}\ar[l]^{\partial_{ext}}  & C_{2,0}\ar[l]^{\partial_{ext}} \\
\ldots\ar[u]  & \ldots\ar[u]  & \ldots\ar[u]   \\
}
\end{equation*}  
For the external differential, there is a contracting homotopy $h: C_{p,q} \to C_{p+1,q-1}$ (for $p \geq 0$ and $q>0$)
$$h((e_0, \ldots, e_{p-1}), (d_0, \ldots, d_{q-1}),y)=((e_0, \ldots, e_{p-1},d_0), (d_1, \ldots, d_{q-1}),y)$$
for which $h \partial_{ext}= \partial_{ext} h +(-1)^p Id$ for $p>0$ and $\partial_{ext}h=Id$ for $p=0$.
This means that the rows in the double complex are completely acyclic, and hence the inclusion of $C_{0,0}$ into the corresponding total complex is a quasi-isomorphism.
But $C_{0,0}$ is exactly the image of $\eta_R$. This proves that $\eta : Y \to \Bar (\coBar (Y))$ is a quasi-isomorphism.

The counit $\varepsilon : \coBar (\Bar (M)) \to M$ can be written in coinvariant notation  as
 $$\varepsilon_R(b|d, s^{-1}(bd|f,sx))=
 \left\{
 \begin{array}{ll}
0 & \text{if } codim(f) > 0 \\
d^*f^*(x) & \text{if } codim(f) = 0
 \end{array}
\right.
 $$
 where $R\to S \to T, b,d,f$ enumerate inner edges of $R$, $S \setminus R$, $T \setminus S$, and $x\in M(T)$, and $d,f$ also denote the morphisms $R\to S$ and $S \to T$.
 For fixed $R$ and $b$, this map $\varepsilon_R$ has a non-natural section $\sigma_R$ defined by 
 $$\sigma_R(x)=(b|-, s^{-1}(b|-,sx))$$
 to which we can apply exactly the same argument as for $\eta$. This shows that for a fixed $R$, $\sigma_R$ and thus $\varepsilon_R$ are quasi-isomorphisms.
\end{proof}

%

\section*{Appendix: Invariants and coinvariants}

The purpose of this appendix is to explain the notation related to
the isomorphism between invariants and coinvariants used for the alternative descriptions of the bar and cobar constructions. 

\bigskip

Let $G$ be a groupoid acting on a family $A= \{A_x ; x \in Ob(G)\}$ in $\Ch$.

There is a canonical map

$$
\left( \bigoplus_{x \in G} A_x \right)_{coinv} 
\stackrel{\rho}{\longrightarrow}
\left( \prod_{x \in G} a_x \right)^{inv}
$$
Writing elements of the group of coinvariants as $[x,a]$ where $x$ is an object of $G$ and $a$ is in $A_x$, the map $\rho$ is defined for $y$ an object of $G$ by $\rho([x,a])_y =0$ if there is no arrow from $x$ to $y$ and by $g \cdot a$ for $g$ a morphism from $x$ to $y$.

Call a groupoid \textit{simple} if there is at most one morphism between any two objects
(ie. $G$ is an ``equivalence relation'').
Since the groups of invariants and of coinvariants are both invariant under categorical equivalence of groupoids, the following proposition is obvious.
In fact, it is enough to ask that $G$ has finitely many connected components.

\begin{prop*}
For a finite simple groupoid $G$, the map $\rho$ is an isomorphism. 
\end{prop*}

This applies to the groupoids $H$ used in the definitions of the bar and cobar constructions. 

As said, we need to be more explicit about the inverse of $\rho$, denoted
$$\int_{x \in G} (-) : 
\left( \prod_{x \in G} a_x \right)^{inv}
\to
\left( \bigoplus_{x \in G} A_x \right)_{coinv} 
$$
and we write
$$\int_{x \in G} \omega = \sum_C \left( \frac{1}{|C|} \sum_{x \in C} [x,\omega(x)] \right) 
$$
(and sometimes simply write $\omega(x)$ for $[x,\omega(x)]$ in this expression).
Here $C$ ranges over the connected components of $G$, and $|C|$ is the number of objects in $C$.    
It looks as if we use characteristic zero in this formula, but it is not really.
The expression $\displaystyle{\frac{1}{|C|} \sum_{x \in C} \omega(x)}$ exists as an element of $\left( \bigoplus_{x \in G} A_x \right)_{coinv} $,
and is represented by $(x, \omega(x))$ for any choice of $x \in C$.

This expression allows us to calculate without making a choice.
For example, if we have a map $f : \displaystyle{ \bigg( \bigoplus_{x \in G} A_x \bigg)_{coinv} \longrightarrow B}$  and $\omega \in \displaystyle{ \bigg( \prod_{x \in G} A_x \bigg)^{inv} }$, 
the formula $\int_G f = \displaystyle{\sum_C \frac{1}{|C|} \sum_{x \in C} f \omega(x)}$ now literally makes sense because $f\omega(x)$ is constant on connected components, and one checks directly that $f=\int_G f \rho$.

%

\bibliographystyle{plain}

\bibliography{biblio}

\end{document}